\documentclass{daj}

\dajAUTHORdetails{%
  title = {On an Almost All Version of the Balog-Szemer\'{e}di-Gowers Theorem}, 
  author = {Xuancheng Shao},
  plaintextauthor = {Xuancheng Shao},
  plaintexttitle = {On an Almost All Version of the Balog-Szemeredi-Gowers Theorem}, 
  runningtitle = {On an Almost All Version of the Balog-Szemer\'{e}di-Gowers Theorem}, 
}   

\dajEDITORdetails{%
   year={2019},
   number={12},
   received={3 December 2018},   
   published={9 September 2019},  
   doi={10.19086/da.9095},       
}   

\usepackage{graphicx}
\usepackage{amssymb}
\usepackage{cite}
\usepackage{amsmath}
\usepackage{latexsym}
\usepackage{amscd}
\usepackage{amsthm}
\usepackage{mathrsfs}
\usepackage{url}
\usepackage[utf8]{inputenc}
\usepackage[english]{babel}
\usepackage{amsfonts}

\def\R{\mathbb R}
\def\Z{\mathbb Z}

\def\F{\mathbb F}
\def\d{\mathrm d}

\def\ee{\varepsilon}
\def\wt{\widetilde}
\def\ve{\mathbf}

\def\vol{\operatorname{vol}}
\def\Int{\operatorname{Int}}

\newtheorem{theorem}{Theorem}[section]
\newtheorem{lemma}[theorem]{Lemma}
\newtheorem{proposition}[theorem]{Proposition}

\newtheorem*{nonumtheorem}{Theorem}
\newtheorem{corollary}[theorem]{Corollary}
\newtheorem{conjecture}[theorem]{Conjecture}

\theoremstyle{remark}
\newtheorem{remark}[theorem]{Remark}

\theoremstyle{definition}

\theoremstyle{remark}

\numberwithin{equation}{section}

\begin{document}

\begin{frontmatter}[classification=text]


\author[xs]{Xuancheng Shao\thanks{Supported by the NSF grant DMS-1802224.}}

\begin{abstract}
We deduce, as a consequence of the arithmetic removal lemma, an almost-all version of the Balog-Szemer\'{e}di-Gowers theorem. It states that for any $K\geq 1$ and $\ee > 0$, there exists $\delta = \delta(K,\ee)>0$ such that the following statement holds: if $|A+_{\Gamma}A| \leq K|A|$ for some $\Gamma \geq (1-\delta)|A|^2$, then there is a subset $A' \subset A$ with $|A'| \geq (1-\ee)|A|$ such that $|A'+A'| \leq |A+_{\Gamma}A| + \ee |A|$. We also discuss the quantitative bounds in this statement, in particular showing that when $A \subset \Z$, the dependence of $\delta$ on $\ee$ cannot be polynomial for any fixed $K>2$.
\end{abstract}
\end{frontmatter}


\section{Introduction}

Let $G$ be an abelian group, and let $A \subset G$ be a finite set. The sumset $A+A$ is defined by
\[ A+A = \{a+a' \colon a,a' \in A\}. \]
A central subject in additive combinatorics is to study the structure of sets $A$ with small sumset. It has emerged that, in some applications, one only has information about a restricted version or a popular-sum version of the complete sumset $A+A$. For $\Gamma \subset A \times A$, we define the restricted sumset
\[ A+_{\Gamma}A = \{a+a' \colon (a,a') \in \Gamma\}. \]
The following natural question arises: if $A+_{\Gamma}A$ is small, can we still obtain structural information on the set $A$? The Balog-Szemer\'{e}di-Gowers theorem (see~\cite[Theorem 2.29]{TV06}) answers this question in the affirmative, by producing a subset $A' \subset A$ with positive density (depending on the density of $\Gamma$ in $A\times A$), such that the complete sumset $A'+A'$ is small.

\begin{nonumtheorem}[Balog-Szemer\'{e}di-Gowers]
Let $G$ be an abelian group, and let $A, B \subset G$ be two subsets. Let $\Gamma \subset A \times B$ be a subset with $|\Gamma| \geq |A||B|/K'$ for some $K' \geq 1$. If $|A+_{\Gamma} B| \leq K|A|^{1/2}|B|^{1/2}$ for some $K \geq 1$, then there exist subsets $A' \subset A$ and $B' \subset B$ such that
\[ |A'| \geq \frac{|A|}{4\sqrt{2}K'}, \ \ |B'| \geq \frac{|B|}{4K'}, \ \ \hbox{and}\ \ |A'+B'| \leq 2^{12} (K')^4 K^3 |A|^{1/2}|B|^{1/2}. \] 
\end{nonumtheorem}

This paper focuses on an ``almost all'' version (or $99\%$ version) of the Balog-Szemer\'{e}di-Gowers theorem. More precisely, if $\Gamma$ in the statement above is almost all (as opposed to just a positive proportion) of $A \times B$, can we take the sets $A',B'$ in the conclusion to be almost all of $A,B$, and moreover can we ensure that the sumset $A'+B'$ is just a little larger than $A+_{\Gamma}B$? We show that the answer to both questions is yes.

\begin{theorem}\label{thm:99bsg}
Let $G$ be an abelian group, and  let $A,B \subset G$ be two subsets with $|A| = |B| = N$. Let $K \geq 1$ and $\ee \in (0,1/2)$, and let $\delta > 0$ be sufficiently small in terms of $K,\ee$. Let $\Gamma \subset A \times B$ be a subset with $|\Gamma| \geq (1-\delta)N^2$. If $|A+_{\Gamma}B| \leq KN$, then there exist subsets $A'\subset A$ and $B' \subset B$ such that
\[ |A'| \geq (1-\ee)N, \ \ |B'| \geq (1-\ee)N, \ \ \hbox{and} \ \ |A'+B'| \leq |A+_{\Gamma}B| + \ee N. \]
Moreover, if $G = \F_p^n$ with $p$ fixed then we may take $\delta = (\ee/K)^{O_p(1)}$.
\end{theorem}

The proof of this theorem is not difficult. In fact, it is closely related to an equivalent formulation of the arithmetic removal lemma; see Proposition~\ref{prop:equiv} below. Consequently, the quantitative dependence of $\delta$ on $K,\ee$ that we are able to obtain in Theorem~\ref{thm:99bsg} is the same as that in the arithmetic removal lemma.

As an immediate consequence of Theorem~\ref{thm:99bsg}, we derive the following structure theorem for sets of integers whose restricted doubling is less than $3$.

\begin{corollary}\label{cor:doubling-3}
Let $A, B\subset \Z$ be two subsets with $|A| = |B| = N$. Let $\ee > 0$, and let $\delta > 0$ be sufficiently small in terms of $\ee$. Let $\Gamma \subset A \times B$ be a subset with $|\Gamma| \geq (1-\delta)N^2$. If $|A+_{\Gamma}B| \leq (3-\ee)N - 4$, then there exist arithmetic progressions $P,Q$ with the same common difference and sizes at most $|A+_{\Gamma}B| - (1-\ee)N + 1$, such that $|A \cap P| \geq (1-\ee)N$ and $|B \cap Q| \geq (1-\ee)N$.
\end{corollary}

By comparison, \cite[Theorem 1.1]{SX18} shows that one can take $\delta = c\ee^2$ for some absolute constant $c>0$, under the more restrictive assumption that $|A+_{\Gamma}B| \leq \left(\tfrac{3+\sqrt{5}}{2} - \ee\right)N$. Can one take $\delta$ to depend polynomially on $\ee$ in Corollary~\ref{cor:doubling-3}? We believe that this should be possible.

\begin{conjecture}\label{conj:doubling-3}
In Corollary~\ref{cor:doubling-3} one can take $\delta = c\ee^C$ for some absolute constants $c,C>0$.
\end{conjecture}

In fact, one may even hope that $\delta = c\ee^2$ should work. It is easy to see that this would be best possible. Indeed, let $\ee > 0$ be small. Take 
\[ A = B = \{n\colon 1 \leq n \leq (1-2\ee)N\} \cup \{n\colon 1.1N \leq n \leq (1.1+2\ee)N\}, \]
and form $\Gamma$ by removing the $(2\ee N)^2$ pairs $(n,n')$ with $1.1N \leq n,n' \leq (1.1+2\ee)N$. Then $|A+_{\Gamma}A| = 2.1N-1$. On the other hand, if $P$ is an arithmetic progression that shares all but at most $\ee N$ elements of $A$, then $P$ must contain all but at most $\ee N$ elements of the interval from $1$ to $(1.1+2\ee)N$, and so $|P| > |A+_{\Gamma}A| - (1-\ee)N + 1$. Hence one must have $\delta \leq O(\ee^2)$ in Corollary~\ref{cor:doubling-3}.

We are unable to settle Conjecture~\ref{conj:doubling-3}, but we show that  Theorem~\ref{thm:99bsg} does not hold with polynomial bounds if $G = \Z$, so that one cannot settle Conjecture~\ref{conj:doubling-3} purely via Theorem~\ref{thm:99bsg}.

\begin{theorem}\label{thm:non-poly}
In Theorem~\ref{thm:99bsg}, if $G = \Z$ then for any $K>2$ and $\ee > 0$ with $D := \ee^{-1}\min(K-2,1)$ sufficiently large, we must have
\[ \delta \leq \exp\left(-c (\log D)^2 (\log\log D)^{-2}\right) \]
for some absolute constant $c>0$.
\end{theorem}

Our construction to prove Theorem~\ref{thm:non-poly} is motivated by, but different from, Behrend's construction of a large set of integers without $3$-term arithmetic progressions. Let $A \subset [M]$ be the $3$-AP-free Behrend set, so that $|A| \geq \exp(-C(\log M)^{1/2}) M$. Let $\Gamma \subset A \times A$ be the set of all non-diagonal pairs (i.e. those pairs $(a,a')$ with $a \neq a'$), so that $\delta = |A|^{-1}$. Then $A+_{\Gamma}A$ misses all elements $2a$ ($a \in A$) by the $3$-AP-free property, so that $|A+_{\Gamma}A| \leq |A+A| - |A|$.  If one removes $\ee |A|$ elements from $A$ to form a subset $A'$, one might guess that $A'+A'$ is smallest if the $\ee |A|$ elements removed are from an initial interval $\{1,2,\cdots,L\}$ for some $L$. In this case, $A'+A'$ should more-or-less be $(A+A)\setminus \{1,2,\cdots, 2L\}$, and so we should have $|A'+A'| \geq |A+_{\Gamma}A| + \ee |A|$ if we take $L = 0.1|A|$ (say). Since $\ee$ is the proportion of elements from $A$ lying in $\{1,2,\cdots, L\}$, it should be of the form $\ee \approx L/M  \approx \exp(-C(\log M)^{1/2})$. Since $\delta \approx 1/M$, this should show that we must have
\[ \delta \leq \exp\left(-c \left(\log \tfrac{1}{\ee}\right)^2 \right) \]
in Theorem~\ref{thm:non-poly} when $G = \Z$. We will make this argument rigorous by constructing $A$ as a discretized version of a thin annulus in $\R^d$ (with large $d$), and we use a trick to make the doubling constant $K$ as close to $2$ as possible. Our analysis leads to an extra  $\log\log\ee^{-1}$ factor, which we do not know how to remove.

\subsection*{Notation} We use $O(X)$ to denote a quantity which is bounded in magnitude by $CX$ for some absolute constant $C>0$. We also write $Y \ll X$ or $X \gg Y$ for the estimate $|Y| = O(X)$.

\section{Preliminaries}

\subsection{Dense models for sets with small doubling}

Recall that a map $\pi \colon G \to \wt{G}$ from one abelian group $G$ to another  is a Freiman isomorphism on $A \subset G$, if it is injective on $A$ and moreover
\[ a_1 + a_2 = a_3 + a_4 \Longleftrightarrow \pi(a_1) + \pi(a_2) = \pi(a_3) + \pi(a_4), \]
for all $a_1,a_2,a_3,a_4 \in A$. In that case we say that $A$ is Freiman isomorphic to $\pi(A)$. Using Freiman's theorem, one can show that any set with small doubling is Freiman isomorphic to a dense set in a finite abelian group (see~\cite[Proposition 1.2]{GR07}).

\begin{proposition}\label{prop:freiman-model}
Let $G$ be an abelian group, and let $A \subset G$ be a finite subset. If $|A+A| \leq K|A|$ for some $K \geq 1$, then one can find a finite abelian group $\wt{G}$ and a subset $\wt{A} \subset \wt{G}$ which is Freiman isomorphic to $A$, with $|\wt{A}| \geq c(K) |\wt{G}|$ for $c(K) = (20K)^{-10K^2}$.
\end{proposition}


In the case $G = \F_p^n$ with $p$ fixed, the following proposition~\cite[Lemma 5.6]{Sis18} says that one can take $c(K)$ to be polynomial in $K$. This is a slight generalization of \cite[Proposition 6.1]{GR07}, and can be proved using Ruzsa's embedding lemma~\cite{Ruz94} (see also~\cite[Lemma 5.26]{TV06}).

\begin{proposition}\label{prop:freiman-model-poly}
If $G = \F_p^n$  then Proposition~\ref{prop:freiman-model} holds with $c(K) = p^{-1}K^{-4}$ and $\wt{G} = \F_p^m$ for some $m$.
\end{proposition}

\subsection{Arithmetic removal lemmas}

The key ingredient in proving Theorem~\ref{thm:99bsg} is the arithmetic removal lemma due to Green~\cite{Gre05} (see also~\cite{KSV09}).

\begin{theorem}\label{thm:arl}
Let $G$ be a finite abelian group, and let $A, B, C \subset G$ be three subsets. Let $\ee > 0$, and let $\delta > 0$ be sufficiently small in terms of $\ee$. If the number of solutions to $a+b=c$ with $a\in A, b \in B, c \in C$ is at most $\delta |G|^2$, then one can remove at most $\ee |G|$ elements from $A,B,C$ to obtain $A',B',C'$, respectively, such that there is no solution to $a+b=c$ with $a \in A',b\in B',c\in C'$.
\end{theorem}

The best known bound in the removal lemma is of tower type; see~\cite{Fox11}. However, in the finite field model $G = \F_p^n$ with $p$ fixed, the removal lemma can be proved with polynomial bounds~\cite{FL17}, building on the recent breakthrough on cap-sets~\cite{CLP17,EG17}.

\begin{theorem}\label{thm:arl-Fp}
If $G = \F_p^n$ with $p$ fixed, then Theorem~\ref{thm:arl} holds with $\delta = (\ee/3)^{C_p}$, where $C_p$ is the constant given by $C_p = 1+c_p^{-1}$ and
\[ p^{1-c_p} = \inf_{0 < x < 1} x^{-(p-1)/3} (1+x+x^2 + \cdots + x^{p-1}). \]
\end{theorem}

Theorem~\ref{thm:arl} is non-trivial only if $A,B,C$ are dense in $G$, but using the dense model theorem (Proposition~\ref{prop:freiman-model}), we can deduce a ``local'' version of the arithmetic removal lemma that applies to sets with small doubling.

\begin{corollary}\label{cor:arl-K}
Let $G$ be an abelian group, let $X \subset G$ be a finite subset with $|X+X| \leq K|X|$ for some $K \geq 1$, and let $A, B, C \subset X$ be three subsets. Let $\ee > 0$, and let $\delta > 0$ be sufficiently small in terms of $K,\ee$. If the number of solutions to $a+b=c$ with $a\in A, b \in B, c \in C$ is at most $\delta |X|^2$, then one can remove at most $\ee |X|$ elements from $A,B,C$ to obtain $A',B',C'$, respectively, such that there is no solution to $a+b=c$ with $a \in A',b\in B',c\in C'$.
\end{corollary}

\begin{proof}
We may assume that $0 \in X$ (at the cost of replacing $K$ by $K+1$). By Proposition~\ref{prop:freiman-model}, there exists a Freiman isomorphism $\phi \colon X \to \wt{X}$, where $\wt{X}$ is a subset of a finite abelian group $\wt{G}$ with $|\wt{X}| \geq c(K) |\wt{G}|$. By an appropriate translation we may assume that $\phi(0) = 0$. Note that
\[ a+b=c+0 \ \ \text{if and only if} \ \ \phi(a) + \phi(b) = \phi(c) + 0 \]
whenever $a,b,c \in X$. Thus the number of solutions to $\phi(a) + \phi(b) = \phi(c)$ with $a \in A,b \in B, c \in C$ is at most $\delta |\wt{X}|^2 \leq \delta |\wt{G}|^2$. By Theorem~\ref{thm:arl} applied to the three sets $\phi(A), \phi(B), \phi(C) \subset \wt{G}$ and with $\ee$ replaced by $\ee c(K)$, one can remove at most $\ee c(K) |\wt{G}| \leq \ee |\wt{X}|$ elements from $A, B, C$ to obtain $A',B',C'$, respectively, such that there is no solution to  $\phi(a) + \phi(b) = \phi(c)$ with $a \in A', b \in B', c \in C'$. This implies that there is also no solution to $a+b=c$ with $a \in A', b \in B', c \in C'$, as desired.
\end{proof}

\begin{corollary}\label{cor:arl-Fp-K}
If $G = \F_p^n$ then Corollary~\ref{cor:arl-K} holds with $\delta = (\ee/3pK^4)^{C_p}$, where $C_p$ is the constant in Theorem~\ref{thm:arl-Fp}.
\end{corollary}

\begin{proof}
Run the same argument as above, but using the polynomial bounds in Proposition~\ref{prop:freiman-model-poly} and Theorem~\ref{thm:arl-Fp}.
\end{proof}

It may be of interest to note that Corollary~\ref{cor:arl-K} can be proved in ``self-contained'' manner, i.e., using triangle removal in graphs but not Freiman's theorem as follows. Apply the standard Kral-Serra-Vena proof~\cite{KSV09} but instead of each part of the graph having vertex set $G$, let it have vertex set $X+X$. Each solution to $a+b=c$ turns into between $|X|$ and $K|X|$ triangles in the graph from which one can complete the proof as usual. However, this proof does not let one apply the recent advances on the arithmetic triangle removal lemma in the finite field setting to derive Corollary~\ref{cor:arl-Fp-K}.

\subsection{A weak version of Theorem~\ref{thm:99bsg}}

\begin{lemma}\label{lem:99BSG-weak}
Let $G$ be an abelian group, and let $A, B\subset G$ be two subsets with $|A| = |B| = N$. Let $\Gamma \subset A \times B$ be a subset with $|\Gamma| \geq (1-\delta)N^2$ for some $0 < \delta < 1/4$. If $|A+_{\Gamma}B| \leq KN$ for some $K \geq 1$, then there exist subsets $A' \subset A$ and $B' \subset B$ such that
\[ |A'| \geq (1- \delta^{1/2})N, \ \ |B'| \geq (1- \delta^{1/2})N,\ \ \hbox{and} \ \ |A'+B'| \leq K^3N / (1-2\delta^{1/2})^2. \]
\end{lemma}

\begin{proof}
When $A=B$ this is~\cite[Lemma 5.1]{GKM15}. For the general case, we consider ``paths of length $3$'' in addition to ``paths of length $2$'', as in~\cite[Section 6.4]{TV06}. Let $A'$ be the set of $a \in A$ such that $(a,b)\in \Gamma$ for at least $(1-\delta^{1/2})N$ elements $b \in B$, and similarly let $B'$ be the set of $b \in B$ such that $(a,b) \in \Gamma$ for at least $(1-\delta^{1/2})N$ elements $a \in A$. Thus
\[ |A'| \geq (1-\delta^{1/2})N\ \ \hbox{and} \ \ |B'| \geq (1-\delta^{1/2})N. \]
It follows that for any $a \in A'$, there are at least $(1-2\delta^{1/2})N$ elements $b \in B'$ such that $(a,b) \in \Gamma$. Thus for any pair $(a',b') \in A' \times B'$, there are at least $(1-2\delta^{1/2})N$ elements $b \in B'$ such that $(a',b) \in \Gamma$, and then for this choice of $b$ there are at least $(1-2\delta^{1/2})N$ values of $a$ such that both $(a,b)$ and $(a,b')$ lie in $\Gamma$. This leads to at least $(1-2\delta^{1/2})^2N^2$ representations of the form
\[ a' + b' = (a' + b) - (a + b) + (a + b') \]
with $(a',b),(a,b),(a,b') \in \Gamma$. It follows that
\[ |A'+B'| \leq \frac{|A+_{\Gamma}B|^3}{(1-2\delta^{1/2})^2N^2} \leq \frac{K^3N}{(1-2\delta^{1/2})^2}.  \]
\end{proof}

\section{Proof of Theorem~\ref{thm:99bsg} and Corollary~\ref{cor:doubling-3}}

The following proposition encapsulates the connection between Theorem~\ref{thm:99bsg} and the arithmetic removal lemma.

\begin{proposition}\label{prop:equiv}
Let $G$ be an abelian group, and let $A,B \subset G$ be two subsets with $|A| = |B| = N$. Let $\ee,\delta > 0$.  The following two statements are equivalent.
\begin{enumerate}
\item  If $\Gamma \subset A \times B$ is a subset with $|\Gamma| \geq (1-\delta)N^2$, then there exist subsets $A'\subset A$, $B' \subset B$ and $S \subset A+B$ with
\[ |A'| \geq (1-\ee)N, \ \ |B'| \geq (1-\ee)N, \ \ \hbox{and}\ \ |S| \leq \ee N, \]
such that $A' + B' \subset (A+_{\Gamma}B) \cup S$.
\item  Let $C \subset G$ be a third subset. If the number of solutions to $a+b=c$ with $a\in A, b \in B, c \in C$ is at most $\delta N^2$, then one can remove at most $\ee N$ elements from $A,B,C$ to obtain $A',B',C'$, respectively, such that there is no solution to $a+b=c$ with $a \in A',b\in B',c\in C'$.
\end{enumerate}
\end{proposition}

\begin{proof}
To show that (2) implies (1), let $C = (A+B) \setminus (A+_{\Gamma}B)$. If $a+b=c$ for some $a\in A, b\in B, c \in C$, then $(a,b) \notin\Gamma$ by the definition of $C$, and thus  the number of solutions to $a+b=c$ with $a \in A, b\in B, c\in C$ is at most $\delta N^2$. Then (2) implies that one can remove at most $\ee N$ elements from $A,B,C$ to obtain $A',B',C'$, respectively, such that $(A'+B') \cap C'$ is empty. Thus
\[ |A'| \geq (1-\ee)N, \ \ |B'| \geq (1-\ee)N,\ \ \hbox{and} \ \ |C \setminus C'| \leq \ee N. \]
Since any element of $A'+B'$  lies in either $A+_{\Gamma}B$ or $C \setminus C'$, (1) follows by taking $S = C \setminus C'$.

To show that (1) implies (2), let $\Gamma = \{(a,b) \in A \times B \colon a+b \notin C\}$. Every pair $(a,b) \notin \Gamma$ leads to a solution to $a+b=c$ with $a \in A, b \in B, c \in C$, so the number of pairs not in $\Gamma$ is at most $\delta N^2$. Then (1) implies that there exist subsets $A' \subset A$, $B' \subset B$ and $S \subset A+B$ with
\[ |A'| \geq (1-\ee)N, \ \ |B'| \geq (1-\ee)N,\ \ \hbox{and} \ \ |S| \leq \ee N, \]
such that $A'+B' \subset (A+_{\Gamma}B) \cup S$. Since $A+_{\Gamma}B$ is disjoint from $C$ by the definition of $\Gamma$, we may take $C' = C \setminus S$ so that $A'+B'$ is disjoint from $C'$, as desired.
\end{proof}

\begin{proof}[Proof of Theorem~\ref{thm:99bsg}]
Suppose that $\delta < 1/100$ is small enough in terms of $K,\ee$. By Proposition~\ref{prop:equiv}, it suffices to prove the second statement of Proposition~\ref{prop:equiv} for any $C \subset G$.  By Lemma~\ref{lem:99BSG-weak}, we can find $A_0 \subset A$ and $B_0 \subset B$ such that
\[ |A_0| \geq (1-\delta^{1/2})N, \ \ |B_0| \geq (1-\delta^{1/2})N,\ \ \hbox{and} \ \ |A_0 + B_0| \leq 2K^3N. \]
By shrinking $A_0$ or $B_0$, we may assume that $|A_0| = |B_0|$. Let $C_0 = C \cap (A_0+B_0)$, and let $X = A_0 \cup B_0 \cup C_0$. Then $X+X$ is contained in the union of the iterated sumsets $nA_0 + mB_0$ with $n,m \in \{0,1,2\}$. Using the Ruzsa triangle inequality, one can deduce that
\[ |X+X| \ll K^{O(1)} N \ll K^{O(1)} |X|. \]
(See~\cite[Corollary 2.24]{TV06}). The number of solutions to $a+b=c$ with $a \in A_0, b \in B_0, c \in C_0$ is at most $\delta N^2 \leq 2\delta  |X|^2$, so Corollary~\ref{cor:arl-K} (applied with $\ee$ replaced by $\ee K^{-O(1)}/2$) implies that one can remove at most $\ee N/2$ elements from $A_0,B_0,C_0$ to obtain $A_0',B_0',C_0'$, respectively, such that there is no solution to $a+b=c$ with $a \in A_0', b \in B_0', c \in C_0'$. If we take $A' = A_0'$, $B' = B_0'$, and $C' = C_0' \cup (C \setminus C_0)$, then there are still no solutions to $a+b=c$ with $a \in A', b \in B', c \in C'$, and moreover
\[ |A \setminus A'| \leq \tfrac{1}{2} \ee N + \delta^{1/2}N \leq \ee N, \ \ |B \setminus B'| \leq \ee N, \ \ \hbox{and}\ \ |C\setminus C'| \leq \ee N, \]
as desired. The polynomial dependence of $\delta$ on $K,\ee$ in the case when $G = \F_p^n$ follows by using Corollary~\ref{cor:arl-Fp-K}.
\end{proof}

\begin{proof}[Proof of Corollary~\ref{cor:doubling-3}]
By Theorem~\ref{thm:99bsg} (applied with $\ee$ replaced by $\ee/10$), there exist subsets $A' \subset A$ and $B' \subset B$ such that
\[ |A'| = |B'| \geq \left(1-\tfrac{\ee}{10}\right)N \ \ \hbox{and}\ \ |A'+B'| \leq |A+_{\Gamma}B| + \tfrac{\ee}{10}N \leq \left(3-\tfrac{9\ee}{10}\right)N - 4. \]
Hence $|A'+B'| \leq 3|A'| - 4$. Apply (an asymmetric version of) Freiman's $3k-4$ theorem (see~\cite{LS95}) to conclude that there exist arithmetic progressions $P,Q$ with the same common differences and sizes at most
\[ |A'+B'| - |A'| + 1  \leq |A+_{\Gamma}B| - (1-\ee)N + 1, \]
such that $A' \subset P$ and $B' \subset Q$. This completes the proof.
\end{proof}

\section{A continuous version of Behrend's construction}

Our construction for proving Theorem~\ref{thm:non-poly} is motivated by Behrend's construction of a large $3$-AP-free set, and in particular motivated by the construction in~\cite{GW10}, which starts with a continuous version and then converts it into a discrete one via a probabilistic argument. We will also start with a continuous set, but to convert it into a discrete set we adopt a more rudimentary approach.

Let $d$ be a positive integer, and henceforth we will always assume that it is sufficiently large. Define
\[ S = \{\ve{x} \in \R^d \colon 1-\eta \leq \|\ve{x}\| \leq 1\},  \]
where $\|\cdot\|$ denotes the $L^2$-norm, $d$ is a large positive integer, and $\eta >0$ is small (say $\eta \leq d^{-10}$). Denote by $V_d$ the volume of the unit ball in $\R^d$. Then the volume of $S$ is
\[ \vol(S) = [1 - (1-\eta)^d] V_d = d\eta V_d (1+O(d\eta)).   \]
In particular $d\eta V_d \ll \vol(S) \ll d\eta V_d$. We will use the crude estimates
\[ d^{-d/2} \leq V_d \leq  10^d \cdot d^{-d/2}. \]

\begin{lemma}\label{lem:3ap-vol}
The volume of the set 
\[ T := \{(\ve{x},\ve{y}) \colon \ve{x},\ve{x}-\ve{y},\ve{x}+\ve{y} \in S\} \]
is $\ll (2\eta)^{d/2-1} \vol(S)^2$.
\end{lemma}

\begin{proof}
If $\ve{x}, \ve{x}-\ve{y}, \ve{x}+\ve{y}$ is a $3$-term progression in $S$, then from the identity
\[ 2\|\ve{x}\|^2 + 2\|\ve{y}\|^2 = \|\ve{x}+\ve{y}\|^2 + \|\ve{x}-\ve{y}\|^2 \]
one can deduce that
\[ \|\ve{y}\|^2 \leq 1-(1-\eta)^2 \leq 2\eta. \]
It follows that  the volume of $T$ is
\[ \ll (2\eta)^{d/2} V_d \cdot \vol(S) \ll (2\eta)^{d/2-1} \vol(S)^2. \]
\end{proof}

Clearly $S+S$ is the ball with radius $2$ centered at the origin, so that $\vol(S+S) = 2^d V_d$.

\begin{proposition}\label{prop:s'+s'}
If $S' \subset S$ is a measurable subset with $\vol(S') \geq (1-\ee)\vol(S)$ for some $\ee \in (0,\eta^{3/2})$, then 
\[ \vol(S'+S') \geq \vol(S+S) - O\left(\frac{20^d}{\eta}\left(\frac{\ee}{\eta} + \ee^{2/3}\right) \vol(S)\right). \]  
In particular, if $\ee = 25^{-d}\eta^3$ then 
\[ \vol(S'+S') \geq \vol(S+S) - \tfrac{1}{100}\vol(S). \]
\end{proposition}

\begin{remark}
It is natural to conjecture that $\vol(S'+S')$ is smallest when $S'$ is the set $\{\ve{x} \in \R^d \colon 1-\eta \leq \|\ve{x}\| \leq 1-\ee\eta\}$, in which case $S'+S'$ is the ball with radius $2 - 2\ee\eta$ and so
\[ \vol(S') \geq (1-O(\ee))\vol(S) \ \  \hbox{and}\ \ \vol(S'+S') \geq \vol(S+S) - O(2^d \ee \vol(S)). \]
Our result is weaker than this, but turns out to be sufficient for our purposes.
\end{remark}

To prove Proposition~\ref{prop:s'+s'}, first we need some estimates on the volume of the sets $R_{\ve{y}} := S \cap (\ve{y} - S)$.

\begin{lemma}\label{lem:ry-bound}
If $\|\ve{y}\| = 2-t$ for some $t \in (0,2)$, then 
\[ \vol(R_{\ve{y}}) \gg V_{d-1}2^{-d} \cdot \eta t^{(d-1)/2} \min(t,\eta). \]
\end{lemma}

\begin{proof}
By symmetry, we may assume that $\ve{y} = (2-t,0,\cdots,0)$. First we show that $R_{\ve{y}}$ contains the set
\[ R_{\ve{y}}^- := \left\{\ve{x} = (x_1,\cdots,x_d) \in S \colon 1 - \tfrac{t}{2} \leq x_1  \leq 1 - \tfrac{t}{2} + \tfrac{\eta}{8}, \ \ 1 - \tfrac{\eta}{2} \leq \|\ve{x}\| \leq 1  \right\}. \]
Clearly $R_{\ve{y}}^-$ is contained in $S$. To see that $R_{\ve{y}}^-$ is also contained in $\ve{y} - S$, take any $\ve{x} \in R_{\ve{y}}^-$ and use  the identity 
\[ \|\ve{y}-\ve{x}\|^2   = \|\ve{x}\|^2 + \|\ve{y}\|^2 - 2\ve{x}\cdot\ve{y}. \]
Since 
\[ 1-\tfrac{\eta}{2} \leq \|\ve{x}\| \leq 1, \ \ \|\ve{y}\| = 2-t,\ \ \hbox{and} \ \ (2-t)^2 \leq 2\ve{x} \cdot \ve{y} \leq (2-t)\left(2-t+\tfrac{\eta}{4}\right), \]
a little bit of algebra reveals that $1-\eta \leq \|\ve{y}-\ve{x}\| \leq 1$, as desired.

Thus it suffices to estimate the volume of $R_{\ve{y}}^-$, which can be done by an integral. For any $x_1 \in (0,1)$, denote by $I(x_1)$ the volume of 
\[ \{\ve{x}' \in \R^{d-1} \colon \ve{x} = (x_1, \ve{x}') \in R_{\ve{y}}^-\}. \]
If $x_1 \leq 1-\eta/2$, then
\[ 
\begin{split}
V_{d-1}^{-1} I(x_1) &= (1-x_1^2)^{(d-1)/2} - \left[\left(1-\tfrac{\eta}{2}\right)^2-x_1^2\right]^{(d-1)/2}  \\
&= (1-x_1^2)^{(d-1)/2} \left[1 - \left(1 - \frac{\eta - \tfrac{\eta^2}{4}}{1-x_1^2}\right)^{(d-1)/2}\right]. \\
\end{split}
\]
Since $1-(1-x)^{(d-1)/2} \geq x$ for any $d \geq 3$ and $x \in [0,1]$, we have
\[ V_{d-1}^{-1} I(x_1) \gg \eta (1-x_1^2)^{(d-3)/2} \gg \eta(1-x_1)^{(d-3)/2} \]
for $x_1 \leq 1 - \eta/2$. On the other hand, if $x_1 \geq 1-\eta/2$, then
\[ V_{d-1}^{-1} I(x_1) = (1-x_1^2)^{(d-1)/2} \gg (1-x_1)^{(d-1)/2}. \]
So overall we always have
\[ I(x_1) \gg \eta (1-x_1)^{(d-1)/2} V_{d-1} \]
for all $x_1 \in (0,1)$. Hence,
\[ 
\begin{split}
\vol(R_{\ve{y}}^-) & = \int_{1-t/2}^{\min(1-t/2+\eta/8, 1)} I(x_1) \d x_1  \\
&\gg \eta V_{d-1} \int_{1-t/2}^{\min(1-t/2+\eta/8, 1)} (1-x_1)^{(d-1)/2} \d x_1 \\
&= \eta V_{d-1} \int_{\max(t/2-\eta/8,0)}^{t/2} x_1^{(d-1)/2} \d x_1 \\
&\gg \eta V_{d-1} 2^{-d} \left(t^{(d+1)/2} - \max\left( (t-\eta/4)^{(d+1)/2}, 0\right) \right). \\
\end{split}
\]
If $t \geq \eta/4$, then
\[ t^{(d+1)/2} - (t-\eta/4)^{(d+1)/2} = t^{(d+1)/2} \left[1 - \left(1 - \tfrac{\eta}{4t}\right)^{(d+1)/2}\right] \gg \eta t^{(d-1)/2}, \]
so that
\[ \vol(R_{\ve{y}}^-) \gg V_{d-1}2^{-d} \cdot \eta^2 t^{(d-1)/2}. \]
If $t \leq \eta/4$, then
\[ \vol(R_{\ve{y}}^-) \gg V_{d-1}2^{-d} \cdot \eta t^{(d+1)/2}. \]
The conclusion follows immediately.
\end{proof}

\begin{remark}
Using Lemma~\ref{lem:ry-bound} one can already get a weaker version of Proposition~\ref{prop:s'+s'}: If $t \geq C\ee^{2/(d+1)}$ for some $C$ large enough in terms of $d$ and (say) $\ee = \eta^3$, then $\ve{y}$ is a popular sum in the sense that $\vol(R_{\ve{y}}) > \ee \vol(S)$, and thus $\ve{y}$ must lie in $S'+S'$. It follows that
\[ \vol(S'+S') \geq \vol(S+S) - O_d\left( \ee^{2/(d+1)} \right). \]
Unfortunately this will not be enough for our purposes because of the exponent $2/(d+1)$ decaying like $1/d$ as $d$ grows. To do better we will make use of the fact that most of the $R_{\ve{y}}$'s are disjoint from each other.
\end{remark}

\begin{lemma}\label{lem:close-y-intersect}
Suppose that $\|\ve{y}_1\| = 2-t_1$ and $\|\ve{y}_2\| = 2-t_2$ for some $t_1,t_2 \in (0,2)$. If $R_{\ve{y}_1} \cap R_{\ve{y}_2}$  is nonempty, then  $\|\ve{y}_1 - \ve{y}_2\| < 2(t_1^{1/2} + t_2^{1/2})$.
\end{lemma}

\begin{proof}
Let $\ve{x} \in R_{\ve{y}_1} \cap R_{\ve{y}_2}$. Then all of $\|\ve{x}\|, \|\ve{y}_1-\ve{x}\|, \|\ve{y}_2-\ve{x}\|$ lie in $[1-\eta, 1]$. From the identity
\[ \|2\ve{x} - \ve{y}_1\|^2 = 2\|\ve{x}\|^2 + 2\|\ve{y}_1-\ve{x}\|^2 - \|\ve{y}_1\|^2, \]
it follows that
\[ \|2\ve{x} - \ve{y}_1\| < 2t_1^{1/2}. \]
Similarly we have
\[ \|2\ve{x} - \ve{y}_2\| < 2t_2^{1/2}, \]
and the conclusion follows by the triangle inequality.
\end{proof}

\begin{proof}[Proof of Proposition~\ref{prop:s'+s'}]
For $t \in (0,2)$, let
\[ D_t := \{\ve{y} \in (S+S) \setminus (S'+S') \colon 2-t \leq \|\ve{y}\| \leq 2-t/2\}. \]
We now obtain an upper bound for $\vol(D_t)$. Pick a maximal set of elements $\ve{y}_1,\cdots,\ve{y}_m \in D_t$ with $R_{\ve{y}_1}, \cdots, R_{\ve{y}_m}$ mutually disjoint. Then for any $\ve{y} \in D_t$, the set $R_{\ve{y}}$ must intersect with some $R_{\ve{y}_i}$, and thus by Lemma~\ref{lem:close-y-intersect}, $\|\ve{y} - \ve{y}_i\| \leq 4t^{1/2}$. It follows that
\[ D_t \subset \bigcup_{1 \leq i \leq m} B(\ve{y}_i, 4t^{1/2}), \]
where $B(\ve{y}_i, 4t^{1/2})$ denotes the ball with radius $4t^{1/2}$ centered at $\ve{y}_i$. By considering the volumes of these balls, we see that $\vol(D_t) \leq m (16t)^{d/2} V_d$. Thus it suffices to bound $m$. Since $\ve{y}_i \notin S'+S'$, each $\ve{x} \in R_{\ve{y}_i}$ must satisfy either $\ve{x} \notin S'$ or $\ve{y}_i - \ve{x} \notin S'$. Hence either $R_{\ve{y}_i} \cap (S \setminus S')$ or $(\ve{y}_i - R_{\ve{y}_i}) \cap (S\setminus S')$ has volume at least $\vol(R_{\ve{y}_i})/2$. But $y_i - R_{\ve{y}_i} = R_{\ve{y}_i}$, so
\[ \vol(R_{\ve{y}_i} \cap (S \setminus S')) \geq \tfrac{1}{2} \vol(R_{\ve{y}_i}). \]
Summing over all $i$, and using the disjointness of $R_{\ve{y}_i}$ and Lemma~\ref{lem:ry-bound}, we get
\[ \ee\cdot d\eta V_d \gg \vol(S\setminus S') \geq \sum_{i=1}^m \vol(R_{\ve{y}_i} \cap (S \setminus S')) \gg m V_{d-1}4^{-d} \cdot \eta t^{(d-1)/2}\min(t,\eta),\]
which leads to the inequalities
\[ m \ll 4^d d\frac{V_d}{V_{d-1}} \cdot \frac{\ee}{t^{(d-1)/2} \min(t,\eta)} \ll 5^d \frac{\ee}{t^{(d-1)/2} \min(t,\eta)}, \]
since $V_d \ll V_{d-1}$. Hence
\[ \vol(D_t) \ll 20^d V_d \frac{\ee t^{1/2}}{\min(t,\eta)}.  \]
We also have the trivial bound
\[ 
\begin{split}
\vol(D_t) &\leq \vol(\{\ve{y} \in \R^d \colon 2-t \leq \|\ve{y}\| \leq 2-t/2\}) \\
& = [(2-t/2)^d - (2-t)^d] V_d \ll d2^dV_d \cdot t, \\
\end{split}
\]
which is better when $t \leq \ee^{2/3}$. Combining these bounds we have
\[ \vol(D_t) \ll 20^d V_d \cdot \begin{cases} \ee t^{1/2}\eta^{-1} & \eta \leq t < 2, \\ \ee t^{-1/2} & \ee^{2/3} \leq t \leq \eta, \\ t & 0 < t \leq \ee^{2/3}. \end{cases} \]
By summing over $t$ dyadically we find that
\[ \vol((S+S)\setminus (S'+S')) \ll 20^d V_d \left(\frac{\ee}{\eta} + \ee^{2/3}\right) \ll \frac{20^d}{\eta} \left(\frac{\ee}{\eta} + \ee^{2/3}\right) \vol(S), \]
as desired.
\end{proof}

While this finishes the analysis of $S$, for technical reasons later on we will need to cut off the corners of $S$ and consider instead the set
\[ \bar{S} = \{\ve{x} = (x_1,\cdots,x_d) \in S \colon \max(|x_1|,\cdots,|x_d|) \leq 1-d^{-10}\}. \]
We summarize the required properties of $\bar{S}$ in the following proposition.

\begin{proposition}\label{prop:S-bar}
Let $d$ be large, and suppose that $\eta \geq d^{-d}$. The set $\bar{S}$ defined above has the following properties.
\begin{enumerate}
\item $\vol(\bar{S}) = d\eta V_d(1 + O(d\eta))$.
\item The set $\bar{T} := \{(\ve{x}, \ve{y}) \colon \ve{x}, \ve{x} - \ve{y}, \ve{x} + \ve{y} \in \bar{S}\}$ has volume $\vol(\bar{T}) \ll (2\eta)^{d/2-1} \vol(\bar{S})^2$.
\item If $S' \subset \bar{S}$ is a measurable subset with $\vol(S') \geq (1 - 30^{-d}\eta^3 )\vol(\bar{S})$, then $\vol(S'+S') \geq \vol(\bar{S}+\bar{S}) - \tfrac{1}{90}\vol(\bar{S})$.
\end{enumerate}

\end{proposition}

\begin{proof}
Note that $S \setminus \bar{S}$ is contained in the union of $2d$ ``caps'' of height $h = d^{-10}$, and each cap has volume
\[ \int_{1-h}^1 (1-x^2)^{(d-1)/2} V_{d-1} \d x \ll 2^d V_{d-1} \int_{1-h}^1 (1-x)^{(d-1)/2} \d x \ll 2^d V_{d-1} h^{(d+1)/2}. \]
Since $V_{d-1}/V_d \ll d^{1/2}$, we have
\[ \vol(S \setminus \bar{S}) \ll 3^d V_d h^{(d+1)/2} \ll d^{-4d} \vol(S).   \]
Hence (1) is clear from the estimate on $\vol(S)$, and (2) follows from Lemma~\ref{lem:3ap-vol}. To prove (3), note that if $S' \subset \bar{S}$ is a subset with $\vol(S') \geq (1-30^{-d}\eta^3) \vol(\bar{S})$, then
\[ \vol(S') \geq (1- 30^{-d}\eta^3 - O(d^{-4d})) \vol(S) \geq (1-25^{-d}\eta^3)\vol(S), \]
where the second inequality follows from the assumption that $\eta \geq d^{-d}$. Hence by Proposition~\ref{prop:s'+s'} we have
\[ \vol(S'+S') \geq \vol(S+S) - \tfrac{1}{100} \vol(S) \geq \vol(\bar{S}+\bar{S}) - \tfrac{1}{90}\vol(\bar{S}). \]
\end{proof}

\section{Non-polynomial bounds: Proof of Theorem~\ref{thm:non-poly}}

Now we discretize $\bar{S}$, the ``trimmed'' version of $S$. Let $d$ be large as in the previous section, and we set $\eta = 2^{-d}$ (say) so that the assumption $\eta \geq d^{-d}$ in Proposition~\ref{prop:S-bar} is satisfied. Let $M$ be sufficiently large in terms of $d,\eta$. We construct a discrete version $A \subset \Z^d$ of $\bar{S}$ as follows. For $\ve{a} \in \Z^d$, denote by $B_M(\ve{a})$ the hypercube $M^{-1}\ve{a} + M^{-1} \cdot [0,1]^d$, which has one corner at $M^{-1}\ve{a}$ and has side length $M^{-1}$. Define
\[ A := \{\ve{a} \in \Z^d \colon B_M(\ve{a}) \subset \bar{S}\}. \]
In particular $A$ lies inside the ball of radius $M$ centered at $0$, and moreover if $\ve{a} = (a_1,\cdots,a_d) \in A$ then $|a_i| \leq (1-d^{-10})M$ for each $i$.

\begin{lemma}\label{lem:A-size}
We have 
\[ (\vol(\bar{S})-o_{M\rightarrow\infty}(1)) M^d \leq |A| \leq \vol(\bar{S}) M^d, \]
where $o_{M\rightarrow\infty}(1)$ denotes a quantity that tends to $0$ as $M \rightarrow\infty$. 
\end{lemma}

\begin{proof}
Since $\bar{S}$ contains $\bigcup_{\ve{a} \in A} B_M(\ve{a})$, we have $|A| \leq \vol(\bar{S}) M^d$. In the other direction, by standard measure theory, $\bigcup_{\ve{a} \in A} B_M(\ve{a})$ is a better and better approximation to $\bar{S}$ as $M\rightarrow \infty$, in the sense that
\[ \vol\left(\bar{S} \setminus \bigcup_{\ve{a} \in A} B_M(\ve{a})\right) = o_{M\rightarrow\infty}(1). \]
This implies the lower bound for $|A|$.
\end{proof}

\begin{lemma}\label{lem:3ap-size}
The number of $3$-term arithmetic progressions in $A$ is $\ll 6^d\eta^{d/2-1}|A|^2$.
\end{lemma}

\begin{proof}
Let $\bar{T}$ be defined as in Proposition~\ref{prop:S-bar}(2). Any $3$-term progression $\ve{a}, \ve{a}-\ve{b}, \ve{a}+\ve{b}$ in $A$ leads to a subset
\[ T_{\ve{a}, \ve{b}} = \{(M^{-1}(\ve{a} + \ve{u}_1), M^{-1}(\ve{b} + \ve{u}_2))\colon \ve{u}_1, \ve{u}_1-\ve{u}_2, \ve{u}_1+\ve{u}_2 \in [0,1]^d  \} \subset \bar{T}. \]
Since the sets $T_{\ve{a}, \ve{b}}$ for different choices of $(\ve{a}, \ve{b})$ are disjoint from each other, and the volume of each $T_{\ve{a},\ve{b}}$ is $(4M^2)^{-d}$, we can conclude that the number of such $(\ve{a}, \ve{b})$ is at most
\[ (4M^2)^d \vol(\bar{T}) \ll 4^d 2^{d/2-1} \eta^{d/2-1} (\vol(\bar{S}) M^d)^2 \ll 6^d \eta^{d/2-1} |A|^2, \]
where we used Lemma~\ref{lem:A-size} in the last inequality.
\end{proof}

\begin{remark}
If one wants to construct a $3$-AP-free set $A$ with this approach, then we would like $6^d \eta^{d/2-1} |A|^2$ to be smaller than $|A|$. Since $|A| \approx M^d$, we would need to require that $M$ is smaller than $\eta^{-1/2}$. This scale is certainly too coarse for the argument to work. However, the construction here is sufficient for the purpose of requiring $A$ to have few $3$-APs rather than none, and it allows us to analyze the sumset $A+A$ (and also  $A'+A'$ with $A'$ almost all of $A$) rigorously.
\end{remark}

For any $\ve{a}, \ve{a}' \in A$,  we have $B_M(\ve{a}+\ve{a}') \subset B_M(\ve{a}) + B_M(\ve{a}') \subset \bar{S}+\bar{S}$. It follows that
\[ M^{-d} |A+A| \leq \vol(\bar{S}+\bar{S}). \]
Now we study $A'+A'$ when $A'$ is almost all of $A$.

\begin{lemma}\label{lem:a'+a'}
If $A' \subset A$ is a subset with $|A'| \geq (1-100^{-d}\eta^3)|A|$, then $|A'+A'| \geq |A+A| - \tfrac{1}{80}|A|$.
\end{lemma}

\begin{proof}
 Let $\Int(A')$ be the set of $\ve{a} \in A'$ with $\ve{a} + \{0,1\}^d \subset A'$ (geometrically think of $\Int(A')$ as the interior of $A'$). Since $A \setminus \Int(A')$ is contained in $A \setminus A' - \{0,1\}^d$, we have
\[ |\Int(A')| \geq (1 - 50^{-d}\eta^3) |A|. \]
Now consider
\[ S' := \bigcup_{\ve{a} \in \Int(A')} B_M(\ve{a}) \subset \bar{S}. \]
We have 
\[ \vol(S') = M^{-d} |\Int(A')| \geq (1-50^{-d}\eta^3) M^{-d} |A| \geq (1-30^{-d}\eta^3) \vol(\bar{S}), \]
where the last inequality folows from Lemma~\ref{lem:A-size}. Hence by Proposition~\ref{prop:S-bar}(3) we have
\[ \vol(S'+S') \geq \vol(\bar{S}+\bar{S}) - \tfrac{1}{90}\vol(\bar{S}). \]
Now note that $S'+S'$ is the union of the hypercubes
\[ B_M(\ve{a}) + B_M(\ve{b}) = \bigcup_{\ve{u} \in \{0,1\}^d} B_M(\ve{a}+\ve{b}+\ve{u}) \]
with $\ve{a}, \ve{b} \in \Int(A')$. But if $\ve{a},\ve{b} \in \Int(A')$, then $\ve{a} + \ve{u} \in A'$, and thus each $\ve{a} + \ve{b} + \ve{u}$ lies in $A'+A'$. Hence we have
\[ S' + S' \subset \bigcup_{\ve{s} \in A'+A'} B_M(\ve{s}), \]
and so 
\[ 
\begin{split}
|A'+A'| &\geq M^d \vol(S'+S') \geq M^d \vol(\bar{S}+\bar{S}) - \tfrac{1}{90} M^d \vol(\bar{S})  \\
& \geq |A+A| - \tfrac{1}{80} |A|.  \\
\end{split}
\]
\end{proof}

Now we can construct our counterexample:

\begin{proposition}\label{prop:non-poly-large-K}
For all sufficiently large $d$, one can find a subset $A \subset \Z$ with $A = -A$ with the following properties.
\begin{enumerate}
\item $A$ lies in an interval of length $d^{O(d)}|A|$.
\item $|A+A| \ll 4^d|A|$.
\item $(A+A)\setminus (A+_{\Gamma}A)$ contains the set $2\cdot A := \{2a \colon a \in A\}$ for some $\Gamma \subset A \times A$ with $|\Gamma| \geq (1-2^{-d^2/3})|A|^2$.
\item $|A'+A'| \geq |A+A| - |A|/80$ for any $A' \subset A$ with $|A'| \geq (1 - 800^{-d})|A|$.
\item For any $a \in A$ there are $\geq d^{-O(d)}|A|$ elements $b \in A$ with $|a-b| \leq 0.1|A|$.
\end{enumerate}
\end{proposition}
 
This already shows that, when $G = \Z$ in Theorem~\ref{thm:99bsg}, $\delta$ cannot depend polynomially on $K/\ee$. More precisely, in Theorem~\ref{thm:99bsg} we must have
\[ \delta \leq \exp\left( -c \left(\log \tfrac{K}{\ee}\right)^2\right) \]
for some absolute constant $c>0$.

\begin{proof}
Let $d$ be large, let  $\eta = 2^{-d}$,  and let $M$ be sufficiently large in terms of $d$. Let $\wt{A} \subset \Z^d$ be the set constructed above. Take $A$ to be the image of $\wt{A}$ under the map  $\pi \colon \Z^d \to \Z$ defined by
\[ \pi(\ve{a}) = a_1 + a_2(10M) + \cdots + a_d(10M)^{d-1} \]
for $\ve{a} = (a_1,\cdots,a_d)$. Since $\wt{A} \subset [-M,M]^d$, $\pi$ is a Freiman isomorphism from $\wt{A}$ to $A$, and thus for the properties (2), (3), and (4), it suffices to prove them with $A$ replaced by $\wt{A}$. 

For (1), $A$ lies in an interval of length 
\[ (10M)^d \ll 10^d \frac{|A|}{\vol(\bar{S})} \ll 10^d \frac{|A|}{d\eta V_d} \leq d^{O(d)}|A|. \]

For (2), we have
\[ |A+A| \leq M^d \vol(\bar{S}+\bar{S}) \leq M^d \cdot 2^d V_d \ll (2M)^d\frac{\vol(S)}{d\eta} \ll \frac{2^d}{d\eta} |A| \ll 4^d |A|. \]

For (3), we define
\[ \Gamma = \{(\ve{a}, \ve{a}') \in \wt{A} \times \wt{A} \colon (\ve{a}+\ve{a}')/2 \notin \wt{A}\}, \]
so that $\wt{A}+_{\Gamma}\wt{A}$ misses the elements $2\ve{a}$ with $\ve{a} \in \wt{A}$. Moreover, by Lemma~\ref{lem:3ap-size}, the number of pairs $(\ve{a},\ve{a}') \in \wt{A} \times \wt{A}$ not in $\Gamma$ is $\ll 6^d \eta^{d/2-1} |A|^2 \leq 2^{-d^2/3}|A|^2$, as desired. 

For (4), it follows immediately from Lemma~\ref{lem:a'+a'} since $100^{-d}\eta^3 = 800^{-d}$.

It remains to establish (5). This is the place where we needed to trim the corners of $S$, so that for any $\ve{a} = (a_1,\cdots,a_d) \in \wt{A}$ we have $|a_i| \leq (1-d^{-10})M$ for each $i$. Let 
\[ a = a_1 + a_2(10M) + \cdots + a_d(10M)^{d-1} \in A. \]
Define $X \subset \bar{S}$ by
\[ X := \{\ve{x} = (x_1,\cdots,x_d) \in \bar{S} \colon |x_d - M^{-1}a_d| \leq (30d)^{-d}\}, \]
and then let 
\[ B := \{\pi(\ve{a}) \colon \ve{a} \in \wt{A}, \ \ B_M(\ve{a}) \cap X \neq \emptyset\}. \]
We claim that all elements $b \in B$ are close to $a$ in the sense that $|a-b| \leq 0.1|A|$. Indeed, if
\[ b = b_1 + b_2(10M) + \cdots + b_d(10M)^{d-1} \in B, \]
then $|a_d - b_d| \leq (30d)^{-d}M+1$ by the definition of $X,B$, and hence
\[ |a-b| \leq (10M)^{d-1} + |a_d-b_d|(10M)^{d-1} \leq 2(10M)^{d-1} + (3d)^{-d} M^d. \]
This is $\leq 0.1|A|$ as claimed, because
\[ |A| \gg \vol(\bar{S})M^d \gg d\eta V_d M^d \geq d\eta d^{-d/2} M^d \geq (2d)^{-d}M^d.  \]
Hence it suffices to show that $|B| \geq d^{-O(d)}|A|$. By the construction of $B$ from $X$, we have $|B| \geq M^d \vol(X)$, and $\vol(X)$ can be estimated via an integral:
\[ 
\begin{split}
\vol(X) &= \int_{M^{-1}a_d - (30d)^{-d}}^{M^{-1}a_d + (30d)^{-d}} \left[(1-x^2)^{(d-1)/2} - ((1-\eta)^2 - x^2)^{(d-1)/2}\right] V_{d-1} \d x \\
&\gg \eta V_{d-1} \int_{M^{-1}a_d - (30d)^{-d}}^{M^{-1}a_d + (30d)^{-d}} (1-x^2)^{(d-3)/2} \d x\\
&\gg \eta V_{d-1} (30d)^{-d} \min_{x\colon |x-M^{-1}a_d| \leq (30d)^{-d}} (1-x^2)^{(d-3)/2}.
\end{split}
\]
Since $M^{-1}a_d \in [-1+d^{-10}, 1-d^{-10}]$, the minimum above is over $x$ with $x \in [-1+d^{-9}, 1-d^{-9}]$, and hence $\vol(X) \gg d^{-O(d)}$. This proves that $|B| \gg d^{-O(d)}M^d \geq d^{-O(d)}|A|$, as desired.
\end{proof}

Finally, we modify the construction above to get counterexamples whose doubling constant can be arbitrarily close to $2$. The following result clearly implies Theorem~\ref{thm:non-poly}, by taking $d =  c\log D/\log\log D$ for some small absolute constant $c>0$.

\begin{proposition}
Let $\lambda \in (0,1/2)$. Let $d$ be large, and set $\delta = 2^{-d^2/3}$, $\ee = \lambda d^{-Cd}$ for some large absolute constant $C>0$. There exists a subset $A \subset \Z$ with $|A+A| \leq (2+\lambda)|A|$, and a subset $\Gamma \subset A \times A$ with $|\Gamma| \geq (1-\delta)|A|^2$, such that 
\[ |A'+A'| > |A+_{\Gamma}A| + \ee |A| \]
for all $A' \subset A$ with $|A'| \geq (1-\ee)|A|$.
\end{proposition}

\begin{proof}
First let $A_0 \subset \Z$ be the set $A$ constructed in Proposition~\ref{prop:non-poly-large-K}. By property (1), $A_0 \subset \{-N,\cdots,N\}$ for some $N = d^{O(d)}|A_0|$, and we may assume that $-N \in A_0$. We will take
\[ A = A_0 \cup \{N+1, \cdots, LN\}, \]
where $L = \lceil 10/\lambda\rceil$. We have
\[ A+A = (A_0 + A_0) \cup \{1, \cdots, 2LN\}. \]
Since $|A| > (L-1)N$ and $|A+A| \leq (2L+2)N+1$, our choice of $L$ ensures that $|A+A| \leq (2+\lambda)|A|$. Let $\Gamma_0 \subset A_0 \times A_0$ be the set from Proposition~\ref{prop:non-poly-large-K}(3). To form $\Gamma$, we remove  from $A \times A$ those pairs in $(A_0 \times A_0)\setminus \Gamma_0$, so that $(A \times A) \setminus \Gamma = (A_0 \times A_0)\setminus \Gamma_0$. Then
\[ |\Gamma| \geq |A|^2 - 2^{-d^2/3}|A_0|^2 \geq \left(1- 2^{-d^2/3}\right) |A|^2 = (1-\delta) |A|^2. \]
Since all non-positive elements in $A+A$ must come from $A_0+A_0$,  we see that $A+_{\Gamma}A$ misses all the elements $2a$ with $a \in A_0$ and $a \leq 0$, by the symmetry of $A_0$ we have
\[ |A+_{\Gamma}A| \leq |A+A| - \tfrac{1}{2} |A_0|. \]
Now let $A' \subset A$ be a subset with $|A'| \geq (1-\ee) |A|$. We can write $A' = A_0' \cup I$, where $A_0' \subset A_0$ and $I \subset \{N+1,\cdots,LN\}$ satisfy the inequalities
\[ |A_0'| \geq |A_0| - \ee|A| \geq (1 - Ld^{O(d)}\ee)|A_0|\ \ \hbox{and}  \ \ |I| \geq (L-1)N - \ee |A|. \]
We need to bound the number of sums in $A+A$ that are missing in $A'+A'$. They come in three types:
\[
\begin{split}
U_1 &:= ((A+A) \setminus (A'+A')) \cap \{-2N, \cdots, 0\}, \\
U_2 &:= ((A+A) \setminus (A'+A')) \cap \{1, \cdots, 2N\}, \\
U_3 &:= ((A+A) \setminus (A'+A')) \cap \{2N+1, \cdots, 2LN\}. \\
\end{split}
\]

To bound $|U_1|$, note that any sum in the range $\{-2N,\cdots,0\}$ must come from $A_0+A_0$, so that $U_1 \subset (A_0+A_0)\setminus (A_0'+A_0')$. By choosing the absolute constant $C$ in the definition $\ee = \lambda d^{-Cd}$ large enough, we can ensure that $Ld^{O(d)}\ee \leq 800^{-d}$, and thus by property (4) of $A_0$ we have
\[ |U_1| \leq |(A_0+A_0) \setminus (A_0'+A_0')| \leq \tfrac{1}{80}|A_0|. \]

To bound $|U_2|$, note that by property (5), there are at least $d^{-O(d)}|A_0| \geq L^{-1}d^{-O(d)}|A|$ elements $b \in A_0$ with $b \leq -N+0.1|A_0|$. Thus after removing at most $\ee |A|$ elements from $A_0$, there still exists some $b \in A_0'$ with $b \leq -N+0.1|A_0|$. Then $b + I \subset A_0'+A_0'$, and $b+I$ covers all but at most $0.1|A_0| + \ee |A|$ elements of $\{1,\cdots,2N\}$. Hence
\[ |U_2| \leq 0.1|A_0| + \ee |A| \leq 0.1|A_0| + Ld^{O(d)}\ee |A_0| \leq 0.15|A_0|. \]

To bound $|U_3|$, note that $I+I \subset \{2N+1,\cdots,2LN\}$ and 
\[ |I+I| \geq 2|I|-1 \geq 2(L-1)N - 2\ee|A| - 1 \geq 2(L-1)N - 3\ee |A|. \]
Hence $|U_3| \leq 3\ee |A| \leq 0.05|A_0|$.

Putting this together, we have
\[ |(A+A) \setminus (A'+A')| = |U_1| + |U_2| + |U_3| < 0.4|A_0|, \]
and hence
\[ |A'+A'| \geq |A+A| - 0.4|A_0| \geq |A+_{\Gamma}A| + 0.1 |A_0| > |A+_{\Gamma}A| + \ee |A|. \]
\end{proof}


\section*{Acknowledgments} 
I would like to thank Ben Green for helpful discussions, and the anonymous referee for valuable suggestions.

\bibliographystyle{amsplain}


\begin{dajauthors}
\begin{authorinfo}[xs]
  Xuancheng Shao\\
  University of Kentucky\\
  Lexington, USA\\
  Xuancheng\imagedot{}Shao\imageat{}uky\imagedot{}edu \\
  \url{http://www.ms.uky.edu/~xsh228/}
\end{authorinfo}
\end{dajauthors}

\end{document}